\documentclass[a4paper,12pt,reqno]{amsart}
\usepackage{eurosym}
\usepackage{amsfonts}
\usepackage{amsmath,amsthm,amssymb}
\usepackage{graphicx, color}

\nonstopmode \numberwithin{equation}{section}

\newtheorem{theorem}{Theorem}
\newtheorem{corollary}{Corollary}[section]

\allowdisplaybreaks

\begin{document}
\title[Unified Integrals associated with generalized Bessel-Maitland function ]
{ Some Unified Integrals associated with generalized Bessel-Maitland function  }
\author{M.S. Abouzaid$^{1}$, A.H. Abusufian$^{2}$, K.S. Nisar$^{2,*}$}
\address{ M.S. Abouzaid: Department of Mathematics, Faculty of Science,
Kafrelshiekh University, Egypt.}
\email{moheb\_abouzaid@hotmail.com}
\address{A.H. Abusufian: $^{1}$College of Arts and Science, Prince Sattam
bin Abdulaziz University, Wadi Al dawaser, Riyadh region 11991, Saudi Arabia}
\email{sufianmath97@hotmail.com}
\address{ K. S. Nisar: Department of Mathematics, College of Arts and
Science, Prince Sattam bin Abdulaziz University, Wadi Al dawaser, Riyadh
region 11991, Saudi Arabia}
\email{ksnisar1@gmail.com}
\subjclass[2010]{33C10, 33C20, 26A33, 33B15}
\thanks{$^{*}$ Corresponding author}
\keywords{Bessel-Maitland function; Generalized Wright hypergeometric
function; Gamma function; Integral formula}

\begin{abstract}
Generalized integral formulas involving the generalized Bessel-Maitland
function are considered and it expressed in terms of generalized Wright
hypergeometric functions. By assuming appropriate values of the parameters
in the main results, we obtained some interesting results of ordinary Bessel
function.
\end{abstract}

\maketitle

\section{Introduction and Preliminaries}

\label{sec-1}

\bigskip The Bessel-Maitland ${J_{v}}^{\alpha }\left( z\right) $ defined by
the following series representation \cite{Marichev} 
\begin{equation}
{J_{v}^{\delta }}\left( z\right) =\sum\limits_{n=0}^{\infty }{\frac{{{{%
\left( {-z}\right) }^{n}}}}{{n!\Gamma \left( {\delta n+v+1}\right) }}}
\label{1}
\end{equation}

The applications of Bessel-Maitland function are found in the field of
applied science, engineering, biological, chemical and physical science is
found in the book of Watson \cite{Watson}.The generalized Bessel-Maitland
function investigated and studied in \cite{Msingh} and defined it as: 
\begin{equation}
{J}_{v,q}^{\delta ,\gamma }\left( z\right) =\sum\limits_{n=0}^{\infty }{%
\frac{{{{\left( \gamma \right) }_{qn}}{{\left( {-z}\right) }^{n}}}}{{%
n!\Gamma \left( {\delta n+v+1}\right) }}}  \label{2}
\end{equation}

where $\mu ,v,\gamma \in C;{\Re }\left( \mu \right) \geq 0,{\Re }\left(
v\right) >-1,{\Re }\left( \gamma \right) \geq 0,q\in \left( {0,1}\right)
\cup \mathbb{N}$ and

\begin{equation}
{\left( \gamma \right) _{0}}=1,{\left( \gamma \right) _{q,n}}=\frac{{\Gamma
\left( {\gamma +qn}\right) }}{{\Gamma \left( \gamma \right) }}  \label{3}
\end{equation}

The generalized hypergeometric function represented as follows \cite%
{Rainville}:
\begin{equation}
_{p}F_{q}\left[ 
\begin{array}{c}
\left( \alpha _{p}\right) ; \\ 
\left( \beta _{q}\right) ;%
\end{array}%
z\right] =\sum\limits_{n=0}^{\infty }\frac{\Pi _{j=1}^{p}\left( \alpha
_{j}\right) _{n}}{\Pi _{j=1}^{q}\left( \beta _{j}\right) _{n}}\frac{z^{n}}{n!%
},  \label{6}
\end{equation}

provided $p\leq q; p=q+1$ and $\left\vert z\right\vert <1$

where $\left( \lambda \right) _{n}$ is well known Pochhammer symbol defined
for $\left( \text{ for }\lambda \in \mathbb{C}\right) $ (see \cite{Rainville}%
)
\begin{equation}
\left( \lambda \right) _{n}:=\left\{ 
\begin{array}{c}
1\text{ \ \ \ \ \ \ \ \ \ \ \ \ \ \ \ \ \ \ \ \ \ \ \ \ \ \ \ \ \ \ \ \ \ \
\ \ }\left( n=0\right) \\ 
\lambda \left( \lambda +1\right) ....\left( \lambda +n-1\right) \text{ \ \ \
\ \ \ \ \ \ \ \ \ \ }\left( n\in \mathbb{N}:=\{1,2,3....\}\right)%
\end{array}%
\right.  \label{7}
\end{equation}

\begin{equation}
\left( \lambda \right) _{n}=\frac{\Gamma \left( \lambda +n\right) }{\Gamma
\left( \lambda \right) }\text{ \ \ \ \ \ \ \ \ \ \ }\left( \lambda \in
C\backslash \mathbb{Z}_{0}^{-}\right) .  \label{8}
\end{equation}

where $\mathbb{Z}_{0}^{-}$ is the set of non positive integers.
The generalized Wright hypergeometric function ${}_{p}\Psi _{q}(z)$ is given
by the series 
\begin{eqnarray}
{}_{p}\Psi _{q}(z) &=&{}_{p}\Psi _{q}\left[ 
\begin{array}{c}
(a_{i},\alpha _{i})_{1,p} \\ 
(b_{j},\beta _{j})_{1,q}%
\end{array}%
\bigg|z\right]  \notag \\
&=&\frac{\prod_{j=1}^{q}\Gamma (\beta _{j})}{\prod_{i=1}^{p}\Gamma (\alpha
_{i})}\displaystyle\sum_{k=0}^{\infty }\dfrac{\prod_{i=1}^{p}\Gamma
(a_{i}+\alpha _{i}k)}{\prod_{j=1}^{q}\Gamma (b_{j}+\beta _{j}k)}\dfrac{z^{k}%
}{k!},  \label{4}
\end{eqnarray}%
where $a_{i},b_{j}\in \mathbb{C}$, and $\alpha _{i},\beta _{j}\in \mathbb{R}$
($i=1,2,\ldots ,p;j=1,2,\ldots ,q$). Asymptotic behavior of this function
for large values of argument of $z\in {\mathbb{C}}$ were studied in \cite%
{Foxc} and under the condition 
\begin{equation}
\displaystyle\sum_{j=1}^{q}\beta _{j}-\displaystyle\sum_{i=1}^{p}\alpha
_{i}>-1  \label{5}
\end{equation}%
was found in the work of \cite{Wright-2,Wright-3}. Properties of this
generalized Wright function were investigated in \cite{Kilbas}, (see also 
\cite{Kilbas-itsf, Kilbas-frac}. In particular, it was proved \cite{Kilbas}
that ${}_{p}\Psi _{q}(z)$, $z\in {\mathbb{C}}$ is an entire function under
the condition ($\ref{5}$).

If we put $\alpha _{1}=...=\alpha _{p}=\beta _{1}=....=\beta _{q}=1$ in ($%
\ref{4}$),then the special case of the generalized Wright function is:
\begin{equation}
{}_{p}\Psi _{q}(z)={}_{p}\Psi _{q}\left[ 
\begin{array}{c}
\left( \alpha _{1},1\right) ,...,\left( \alpha _{p},1\right) ; \\ 
\left( \beta _{1},1\right) ,...,\left( \beta _{q},1\right) ;%
\end{array}%
z\right] =\dfrac{\prod_{j=1}^{q}\Gamma (\beta _{j})}{\prod_{i=1}^{p}\Gamma
(\alpha _{i})}\text{ }_{p}F_{q}\left[ 
\begin{array}{c}
\alpha _{1},...,\alpha _{p}; \\ 
\beta _{1},...,\beta _{q};%
\end{array}%
z\right]  \label{9}
\end{equation}

provided $0<\Re \left( \mu \right) <\Re \left( \lambda \right) .~$Here our
main aim to establish two generalized integral formulas, which are expressed
in terms of generalized Wright functions by inserting generalized
Bessel-Maitland function.

The unified integral formulas involving special functions attract the
attention many authors in eighteenth century (see, \cite{Brychkov},\cite%
{Choi1}). In 1888, Pincherle studied the integrals involving product of
gamma functions along vertical lines (see \cite%
{Pincherle,Pincherle1,Pincherle2}).The extension of such study carried by
Barnes \cite{Barnes}, Mellin \cite{Mellin} and Cahen \cite{Cahen} and
applied some of these integrals in the study of Riemann zeta function and
other Drichlet's series. The unified integral representation of the Fox $H$%
-functions given by Garg \cite{Garg} and the unified integral representation
of the hypergeometric $_{2}F_{1}$ functions by Ali \cite{Ali} . Recently,
Choi and Agarwal \cite{Choi2} obtained two unified integral representations
of Bessel functions $J_{v}\left( z\right) .~$For more \ details are found in
various recent works (see \cite{Choi-Mathur},\cite{Nisar1},\cite{Nisar2},%
\cite{Nisar3}, \cite{Manaria}, \cite{Maneria2}). For the present
investigation, we need the following result of Oberhettinger \cite{Ober}

\begin{equation}
\int_{0}^{\infty }x^{\mu -1}\left( x+a+\sqrt{x^{2}+2ax}\right) ^{-\lambda
}dx=2\lambda a^{-\lambda }\left( \frac{a}{2}\right) ^{\mu }\frac{\Gamma
\left( 2\mu \right) \Gamma \left( \lambda -\mu \right) }{\Gamma \left(
1+\lambda +\mu \right) }  \label{10}
\end{equation}

\section{Main results}

\begin{theorem}
\label{Th1} For $\delta ,\lambda ,v,c\in \mathbb{C}$,then the following
integral formula hold true:%
\begin{eqnarray*}
&&\int_{0}^{\infty }x^{\mu -1}\left( x+a+\sqrt{x^{2}+2ax}\right) ^{-\lambda }%
{J}_{v,q}^{\delta ,\gamma }\left( \frac{y}{x+a+\sqrt{x^{2}+2ax}}\right) dx \\
&=&\frac{\Gamma \left( 2\mu \right) }{\Gamma \left( \gamma \right) }a^{\mu
-\lambda }2^{1-\mu } \\
&&\times _{3}\Psi _{3}\left[ 
\begin{array}{c}
\left( \gamma ,q\right) ,\left( \lambda +1,1\right) ,\left( \lambda -\mu
,1\right) ; \\ 
\left( \nu +1,\delta \right) ,\left( 1+\lambda +\upsilon +\mu ,1\right)
,\left( \lambda ,1\right) ;%
\end{array}%
-\frac{y}{a}\right] 
\end{eqnarray*}
\end{theorem}

\begin{proof}
By applying $\left( \ref{2}\right) $ to the LHS of theorem $\ref{Th1}~$and
interchanging the order of integration and summation, which is verified by
uniform convergence of the involved series under the given conditions , we
obtain%
\begin{eqnarray*}
&&\int_{0}^{\infty }x^{\mu -1}\left( x+a+\sqrt{x^{2}+2ax}\right) ^{-\lambda }%
{J}_{v,q}^{\delta ,\gamma }\left( \frac{y}{x+a+\sqrt{x^{2}+2ax}}\right) dx \\
&=&\int_{0}^{\infty }x^{\mu -1}\left( x+a+\sqrt{x^{2}+2ax}\right) ^{-\lambda
} \\
&&\times \sum\limits_{n=0}^{\infty }{\frac{{{{\left( \gamma \right) }_{qn}}{{%
\left( {-y}\right) }^{n}}}}{{n!\Gamma \left( {\delta n+v+1}\right) }}}\left(
x+a+\sqrt{x^{2}+2ax}\right) ^{-n}dx \\
&=&\sum\limits_{n=0}^{\infty }{\frac{{{{\left( \gamma \right) }_{qn}}{{%
\left( {-y}\right) }^{n}}}}{{n!\Gamma \left( {\delta n+v+1}\right) }}}%
\int_{0}^{\infty }x^{\mu -1}\left( x+a+\sqrt{x^{2}+2ax}\right) ^{-\left(
\lambda +n\right) }dx
\end{eqnarray*}

Since $\mathbb{R}\left( \lambda \right) >\mathbb{R}\left( \mu \right) >0,$ $%
k\in \mathbb{N}_{0}:=\mathbb{N}\cup \left\{ 0\right\} .$

Applying $\left( \ref{10}\right) $ to the integrand of theorem $\ref{Th1}$\
and we obtain the following expression:%
\begin{eqnarray*}
&=&\sum\limits_{n=0}^{\infty }{\frac{{{{\left( \gamma \right) }_{qn}}{{%
\left( {-y}\right) }^{n}}}}{{n!\Gamma \left( {\delta n+v+1}\right) }}}%
2\left( \lambda +n\right) a^{-\left( \lambda +n\right) }\left( \frac{a}{2}%
\right) ^{\mu } \\
&&\times \frac{\Gamma \left( 2\mu \right) \Gamma \left( \lambda +n-\mu
\right) }{\Gamma \left( 1+\lambda +\mu +n\right) }
\end{eqnarray*}

By making the use of the equation $\left( \ref{3}\right) $, we have%
\begin{eqnarray*}
&=&2^{1-\mu }a^{\mu -\lambda }\frac{\Gamma \left( 2\mu \right) }{\Gamma
\left( \gamma \right) } \\
&&\times \sum_{k-0}^{\infty }\frac{\Gamma \left( \gamma +qn\right) \Gamma
\left( \lambda +n-\mu \right) \Gamma \left( \lambda +1+n\right) }{{n!\Gamma
\left( {\delta n+v+1}\right) }\Gamma \left( \lambda +n\right) \left(
1+\lambda +\mu +n\right) } \\
&&\times \left( -\frac{y}{a}\right) ^{n}
\end{eqnarray*}

In view of $\left( \ref{9}\right) $, we obtain the desired result.
\end{proof}

\begin{corollary}
\label{Cor1} If we set $q=0$ and $\gamma =1$, then we obtain integral
representation of ordinary Bessel-Maitland function $\left( \ref{1}\right) $
as:%
\begin{eqnarray*}
&&\int_{0}^{\infty }x^{\mu -1}\left( x+a+\sqrt{x^{2}+2ax}\right) ^{-\lambda }%
{J}_{v}^{\delta }\left( \frac{y}{x+a+\sqrt{x^{2}+2ax}}\right) dx \\
&=&\Gamma \left( 2\mu \right) a^{\mu -\lambda }2^{1-\mu } \\
&&\times _{2}\Psi _{3}\left[ 
\begin{array}{c}
\left( \lambda +1,1\right) ,\left( \lambda -\mu ,1\right) ; \\ 
\left( \nu +1,\delta \right) ,\left( 1+\lambda +\upsilon +\mu ,1\right)
,\left( \lambda ,1\right) ;%
\end{array}%
-\frac{y}{a}\right] 
\end{eqnarray*}
\end{corollary}

\begin{corollary}
\label{Cor2} If we set $q=0$ and $\delta =$ $\gamma =1$, then we obtain
integral representation of ordinary Bessel-Maitland function $\left( \ref{1}%
\right) $ as:%
\begin{eqnarray*}
&&\int_{0}^{\infty }x^{\mu -1}\left( x+a+\sqrt{x^{2}+2ax}\right) ^{-\lambda }%
{J}_{v}\left( \frac{y}{x+a+\sqrt{x^{2}+2ax}}\right) dx \\
&=&\Gamma \left( 2\mu \right) a^{\mu -\lambda }2^{1-\mu } \\
&&\times _{2}\Psi _{3}\left[ 
\begin{array}{c}
\left( \lambda +1,1\right) ,\left( \lambda -\mu ,1\right) ; \\ 
\left( \nu +1,1\right) ,\left( 1+\lambda +\upsilon +\mu ,1\right) ,\left(
\lambda ,1\right) ;%
\end{array}%
-\frac{y}{a}\right] 
\end{eqnarray*}
\end{corollary}

\begin{theorem}
\label{Th2} For $\delta ,\lambda ,v,c\in \mathbb{C}$,then the following
integral formula hold true:%
\begin{eqnarray*}
&&\int_{0}^{\infty }x^{\mu -1}\left( x+a+\sqrt{x^{2}+2ax}\right) ^{-\lambda }%
{J}_{v,q}^{\delta ,\gamma }\left( \frac{xy}{x+a+\sqrt{x^{2}+2ax}}\right) dx
\\
&=&2^{1-\mu }a^{\mu -\lambda }\frac{\Gamma \left( \lambda -\mu \right) }{%
\Gamma \left( \gamma \right) } \\
&&\times _{3}\Psi _{3}\left[ 
\begin{array}{c}
\left( \gamma ,q\right) ,\left( \lambda +1,1\right) ,\left( 2\mu ,2\right) ;
\\ 
\left( \nu +1,\delta \right) ,\left( 1+\lambda +\upsilon +\mu ,2\right)
,\left( \lambda ,1\right) ;%
\end{array}%
-\frac{y}{a}\right] 
\end{eqnarray*}
\end{theorem}

\begin{proof}
By applying $\left( \ref{2}\right) $ to the LHS of theorem $\ref{Th2}~$and
interchanging the order of integration and summation, which is verified by
uniform convergence of the involved series under the given conditions , we
obtain%
\begin{eqnarray*}
&&\int_{0}^{\infty }x^{\mu -1}\left( x+a+\sqrt{x^{2}+2ax}\right) ^{-\lambda }%
{J}_{v,q}^{\delta ,\gamma }\left( \frac{xy}{x+a+\sqrt{x^{2}+2ax}}\right) dx
\\
&=&\int_{0}^{\infty }x^{\mu -1}\left( x+a+\sqrt{x^{2}+2ax}\right) ^{-\lambda
} \\
&&\times \sum\limits_{n=0}^{\infty }{\frac{{{{\left( \gamma \right) }_{qn}}{{%
\left( {-xy}\right) }^{n}}}}{{n!\Gamma \left( {\delta n+v+1}\right) }}}%
\left( x+a+\sqrt{x^{2}+2ax}\right) ^{-n}dx \\
&=&\sum\limits_{n=0}^{\infty }{\frac{{{{\left( \gamma \right) }_{qn}}{{%
\left( {-y}\right) }^{n}}}}{{n!\Gamma \left( {\delta n+v+1}\right) }}}%
\int_{0}^{\infty }x^{\mu +n-1}\left( x+a+\sqrt{x^{2}+2ax}\right) ^{-\left(
\lambda +n\right) }dx
\end{eqnarray*}

Since $\mathbb{R}\left( \lambda \right) >\mathbb{R}\left( \mu \right) >0,$ $%
k\in \mathbb{N}_{0}:=\mathbb{N}\cup \left\{ 0\right\} .$

Applying $\left( \ref{10}\right) $ to the integrand of theorem \ref{Th2}\
and we obtain the following expression:%
\begin{eqnarray*}
&=&\sum\limits_{n=0}^{\infty }{\frac{{{{\left( \gamma \right) }_{qn}}{{%
\left( {-y}\right) }^{n}}}}{{n!\Gamma \left( {\delta n+v+1}\right) }}}%
2\left( \lambda +n\right) a^{-\left( \lambda +n\right) }\left( \frac{a}{2}%
\right) ^{\mu +n} \\
&&\times \frac{\Gamma \left( 2\left( \mu +n\right) \right) \Gamma \left(
\lambda +n-\mu -n\right) }{\Gamma \left( 1+\lambda +\mu +n+n\right) }
\end{eqnarray*}

By making the use of the equation $\left( \ref{3}\right) $, we have%
\begin{eqnarray*}
&=&2^{1-\mu }a^{\mu -\lambda }\frac{\Gamma \left( \lambda -\mu \right) }{%
\Gamma \left( \gamma \right) } \\
&&\times \sum_{k-0}^{\infty }\frac{\Gamma \left( \gamma +qn\right) \Gamma
\left( \lambda +1+n\right) \Gamma \left( 2\mu +2n\right) }{{n!\Gamma \left( {%
\delta n+v+1}\right) }\Gamma \left( \lambda +n\right) \Gamma \left(
1+\lambda +\mu +2n\right) } \\
&&\times \left( -\frac{y}{2}\right) ^{n}
\end{eqnarray*}

In view of $\left( \ref{9}\right) $, we obtain the desired result.
\end{proof}

\begin{corollary}
If we set $q=0$ and $\gamma =1$, then we obtain integral representation of
ordinary Bessel-Maitland function $\left( \ref{1}\right) $ as:%
\begin{eqnarray*}
&&\int_{0}^{\infty }x^{\mu -1}\left( x+a+\sqrt{x^{2}+2ax}\right) ^{-\lambda }%
{J}_{v}^{\delta }\left( \frac{xy}{x+a+\sqrt{x^{2}+2ax}}\right) dx \\
&=&2^{1-\mu }a^{\mu -\lambda }\Gamma \left( \lambda -\mu \right)  \\
&&\times _{2}\Psi _{3}\left[ 
\begin{array}{c}
\left( \lambda +1,1\right) ,\left( 2\mu ,2\right) ; \\ 
\left( \nu +1,\delta \right) ,\left( 1+\lambda +\upsilon +\mu ,2\right)
,\left( \lambda ,1\right) ;%
\end{array}%
-\frac{y}{a}\right] 
\end{eqnarray*}
\end{corollary}

\begin{corollary}
If we set $q=0$ and $\delta =\gamma =1$, then we obtain integral
representation of Bessel function ${J}_{v}\left( z\right) $ as:%
\begin{eqnarray*}
&&\int_{0}^{\infty }x^{\mu -1}\left( x+a+\sqrt{x^{2}+2ax}\right) ^{-\lambda }%
{J}_{v}\left( \frac{xy}{x+a+\sqrt{x^{2}+2ax}}\right) dx \\
&=&2^{1-\mu }a^{\mu -\lambda }\Gamma \left( \lambda -\mu \right)  \\
&&\times _{2}\Psi _{3}\left[ 
\begin{array}{c}
\left( \lambda +1,1\right) ,\left( 2\mu ,2\right) ; \\ 
\left( \nu +1,1\right) ,\left( 1+\lambda +\upsilon +\mu ,2\right) ,\left(
\lambda ,1\right) ;%
\end{array}%
-\frac{y}{a}\right] 
\end{eqnarray*}
\end{corollary}

\textbf{Conclusion: }

Certain unified integral representation of Bessel-Maitland function and its
special cases are derived in this study. In this sequel, one can obtain
integral representation of more generalized special functions, which has
much application in physics and engineering Science.

\end{document}